\documentclass[12pt]{amsart}
\usepackage{amsfonts,amsmath,amsthm,oldgerm,amssymb,amscd}

\usepackage[all,cmtip]{xy}

\usepackage{fancyhdr}
\usepackage{psfrag}
\usepackage{graphicx}

\usepackage[indent]{caption}
\usepackage{indentfirst}
\usepackage{float}
\usepackage{latexsym}
\usepackage{graphics}
\usepackage{verbatim}

\newcommand{\Q}{{\mathbb Q}}

\newcommand{\Z}{{\mathbb Z}}
\newcommand{\C}{{\mathbb C}}

\newcommand{\G}{{\mathbb G}}

\newtheorem{theorem}{Theorem}
\newtheorem{lemma}{Lemma}
\newtheorem{corollary}{Corollary}
\newtheorem{proposition}[theorem]{Proposition}
\numberwithin{theorem}{section}
\numberwithin{corollary}{section}
\numberwithin{lemma}{section}

\theoremstyle{definition}

\numberwithin{conj}{section}
\newtheorem{example}{Example}
\newtheorem*{acknowledgement}{Acknowledgement}
\numberwithin{example}{section}

\numberwithin{definition}{section}

\numberwithin{question}{section}
\numberwithin{equation}{section}

\theoremstyle{remark}
\newtheorem{remark}{Remark}
\numberwithin{remark}{section}

\begin{document}

\title[Locally potentially equivalent Galois representation]{On the
  structure of locally \\
potentially equivalent Galois representations}

\author{Vijay M. Patankar}
\address{Department of Mathematics, BITS Pilani, Goa campus,
  Zuarinagar, Goa, India 403726.\\
email: vijaypatankar@gmail.com}  
\author{C. S. Rajan}
\address{School of Mathematics, Tata Institute of Fundamental
  Research, Dr. Homi Bhabha Road, Colaba, Bombay, INDIA 400005.\\
email: rajan@math.tifr.res.in}


\begin{abstract} Suppose $\rho_1, \rho_2$ are 
two $\ell$-adic Galois representations of the absolute Galois group of
a number field, such 
that the algebraic monodromy group of one of the 
representations is connected and the representations 
are locally potentially equivalent at a
set of places of positive upper density. We classify such pairs of
representations and show that up to twisting by some
 representation, it is given by a pair of representations
one of which is trivial and the other abelian. 

Consequently, assuming that the first representation has connected
algebraic monodromy group, we obtain 
that the representations are  potentially equivalent, 
provided one of the following conditions hold: (a) the first
representation is absolutely irreducible; (b) the
ranks of the algebraic monodromy groups are equal; (c)
the algebraic monodromy group of the second 
representation is also connected and (d) the commutant of the image of the second representation remains the same upon restriction to subgroups of finite index of the Galois group. 

\end{abstract}

\maketitle

\section{Introduction}
Two $n$-dimensional linear representations $\rho_1, ~\rho_2$ of a
group  $\Gamma$ are said to be {\em potentially equivalent} if they
become isomorphic upon restriction to a subgroup of finite index in
$\Gamma$. In this article, we are interested in proving potential
equivalence  of $\ell$-adic representations of the absolute Galois group of a
global field, which are  locally potentially equivalent at a
sufficiently large number of places of the global field.

For a field $L$, denote by $G_L:={\rm Gal}(\bar{L}/L)$ an absolute
Galois group of $L$, where $\bar{L}$ denote a seperable algebraic
closure of $L$. Let  $K$ be a global field, i.e., 
a number field or the function field
of a curve over a finite field. For a place $v$ of $K$
let $K_v$ denote the completion of $K$ at $v$.  Choosing a place $w$
of $\bar{K}$ lying above $v$, allows us to identity $G_{K_v}$ with the
decomposition subgroup $D_w$ of $G_K$. As $w$ varies this gives a
conjugacy class of subgroups of $G_K$. Given a representation $\rho$
of $G_K$ as above, define the localization (or the local component) \(
\rho_v \) of $\rho$ at $v$, to be the representation of $G_{K_v}$
obtained by restricting $\rho$ to a decomposition subgroup. This is
well defined upto isomorphism.

Let \( F \) be a local field of characteristic zero and residue
characteristic \( \ell \) relatively prime to the characteristic of \(
K \).  Suppose $\rho_i: G_K \to GL_n(F), ~i=1,2$ are two continuous
semisimple \( \ell \)-adic representations of $G_K$,   unramified
outside a finite set $\Sigma_{ram}$ of places of $K$ containing the
archimedean places of $K$.  Let $T$ be a set of places of $K$.  Define
$\rho_1$ and $\rho_2$ to be {\em locally potentially equivalent at
$T$}, if for each $v\in T$, the restrictions $\rho_{1, v}$ and
$\rho_{2, v}$ to $G_{K_v}$ are potentially equivalent. 
 
 Given a continuous 
representation  $\rho: ~G_K\to GL_n(F)$ and a place
$v $ of $K$ where $\rho$ is unramified, let $\rho(\sigma_v)$ denote
the Frobenius conjugacy class in the image group $  G_K/{\rm
Ker}(\rho) \simeq \rho( G_K) \subset GL_n(F)$.  By an abuse of
notation, we will also continue to denote by \( \rho (\sigma_v) \) an
element in the associated conjugacy class. Assume further that the 
elements $ \rho_1 (\sigma_v)$ and $
\rho_2(\sigma_v)$ are semisimple (by Corollary \ref{semisimple}, 
this property is satisfied for semisimple representations at a set
of places of density one). The representations
$\rho_1$ and $\rho_2$ are locally potentially equivalent at a prime
$v$ not in $\Sigma_{ram}$,  precisely when the eigenvalues of the
Frobenius conjugacy classes $ \rho_1 (\sigma_v)$ and $
\rho_2(\sigma_v)$ differ by roots of unity.  

For a representation $\rho: G_K \to GL_n(F)$, let \( G_{\rho} \)  be
the algebraic  monodromy  group attached to \( \rho \) over \( F \),
i.e.,  the smallest algebraic subgroup $G_{\rho}$ of $GL_n$ defined
over $F$ such that $\rho(G_K)\subset G_{\rho}(F)$.  Denote by \( G_i
\) the algebraic monodromy groups associated to the representations \(
\rho_i \) for \( i = 1, 2 \). 

In \cite{Ra}, it is proved that if the representations 
$\rho_1$ and $\rho_2$ are locally equivalent (in fact, enough to
assume only that the character values agree evaluated on the Frobenius element)
 at a set of  unramified places having
positive upper density and the algebraic monodromy group of one of the
representations is connected, then 
$\rho_1$ and $\rho_2$ are   potentially equivalent. 
We recall the upper density of a set \( S \) of finite places 
of \( K \) is defined as:
\[ 
ud (S) := \limsup_{x \to \infty } \# \{ v \in S | ~ N v \leq x \}/ \pi (x) , 
\] 
where \( \pi (x ) \) is the number of finite places \( v \) of \( K
\) with  \( N v \leq x \). Here  \( N v \) denotes the cardinality
of the residue field $k_v$ of \( K_v \). 

Based on this result, we consider the following generalization:
assume that the algebraic monodromy group of either $\rho_1$ or
$\rho_2$ is connected and the set of places $v$ of $K$ where the
localizations  $\rho_{1,v}$ and $\rho_{2,v}$  are potentially
equivalent has positive upper density.  
Then are $\rho_1$ and $\rho_2$ (globally)  potentially equivalent? 

One of the motivations to consider this question is to understand the
distribution of Frobenius fields of elliptic curves. 
Let \(A\) be an abelian variety defined over  a number field \(K\).
 The Galois group $G_K$ acts in a natural manner on $A(\bar{K})$.
For a rational prime $\ell$, the Tate module
$T_\ell(A):=\varprojlim_n A[\ell^n]$ is the $G_K$-module obtained
as a projective  limit of the  $G_K$-modules $A[\ell^n]$ of
$\ell^n$-torsion points of $A$ over $\bar{K}$.  Let
$V_{\ell}(A)=T_\ell (A)\otimes_{\Z_l}\Q_l$. The Tate
module is of rank $2d$ over the ring of $\ell$-adic integers
$\Z_\ell$, where $d$ is the dimension of $A$. 

When $A$ is an elliptic curve defined over $K$, we have a  continuous $\ell$-adic representation
$\rho_{A,\ell}:G_K\to GL_2(\Q_\ell)$. In \cite{KPR}, 
the above question is answered in the affirmative  for the 
representations of $G_K$ acting on the Tate module of non-CM elliptic
curves. At an unramified finite  place
$v$ of $K$ for $A$,  
the Frobenius field of $A$ at $v$ is the subfield of algebraic numbers
generated by the eigenvalues of the Frobenius at $v$.   As an
application, it is shown that if the set of places of a number field 
$K$ at which the
Frobenius fields of two non-CM elliptic curves defined over $K$ are
isomorphic has positive upper density, then the curves are isogenous
over a finite extension of $K$. 

More generally, one can try to
understand the multiplicative identities satisfied by the eigenvalues
of a pair of Galois representations, say the natural
representations on the Tate module attached to a pair of abelian
varieties. We hope that the method outlined out here, might shed some
light on the questions raised in  (\cite{Ko}).

Coming back to the question raised above, it was claimed in (\cite{PR}) that 
if the algebraic monodromy group of one of the
representations is connected 
and the upper density of the set $T$ of places of $K$ at
which \( \rho_1 \) and \( \rho_2 \) are locally potentially 
equivalent is positive, then
the representations  $\rho_1$ and $\rho_2$
are potentially equivalent. 
It was pointed out by J.-P. Serre that this claim is wrong. The
proof given in (\cite{PR}) however 
goes through if we assume that the density of $T$ is one. 

The problem was that in the course of the proof (see Remark
\ref{remark:error}), 
the following assertion was made: let $G$
be a reductive algebraic group, and let $m$ be a number divisible by
the exponent of the group $G/G^0$. Then for any connected component
$G^{\phi}$ of $G$, the $m$-th power map $G^{\phi}\to G^0$ is
surjective, where $G^0$ is the connected component of identity in $G$. 

\begin{example}\label{example-serre}
Serre gave the following counter-example: Let $G$ be the
normalizer of the diagonal torus in $\textrm{SL}(2)$. The square map
sends the non-identity component to scalar matrices ($\pm \textrm{Id}
$) and is not surjective.  In terms of
representations, let $p_1$ denote the trivial two dimensional representation 
and $p_2$ the natural two dimensional representation of $G$. 
The characteristic polynomials of both these representations  are
equal evaluated at the 
fourth power of elements of the non identity connected component of
$G$. 
\end{example}

\begin{example} \label{Serre-arithexample}
Based on the above example, an arithmetical counterexample can be
given to the claim made in \cite{PR}.  Let $L=\Q(i)$ be the field of
Gaussian numbers, and fix an embedding  of $L$ into $\C$. Let $\ell$ be a rational prime that splits completely in $L$, and choose a prime $w$ of $L$ dividing $\ell$.  Associated
to an elliptic curve $E$ with  complex multiplication by $\Z[i]$,
there exists, by the theory of complex multiplication,  a continuous
character  $\psi: G_L\to L_w^*$ 
with the following property: at a finite  place $v$ of $L$ with
residue field characteristic coprime to $\ell$ and
unramified for $E$,  the Frobenius element $\pi_v:=\psi(\sigma_v)$ lies in 
$\Z[i]$, and generates a prime ideal in $\Z[i]$. Further
$\psi(\sigma_v)\psi^{\tau}(\sigma_v)$ is equal to $p$ (resp. to
$p^2$), when $Nv=p$ (resp. $Nv=p^2$), where $v$ divides the rational
prime $p$ of $\Q$ and  $\tau$ denotes  the non-trivial element of
${\rm Gal}(L/\Q)$. 

Let $\chi: G_{\Q}\to \Q_{\ell}^*$ denote the $\ell$-adic cyclotomic
character. We have for $p$ coprime to $\ell$,
$\chi(\sigma_p)=p$. Consider the two representations, 
\[ \rho_1=\chi\oplus \chi \quad \mbox{and}\quad \rho_2={\rm
  Ind}_{G_K}^{G_{\Q}}(\psi^2). \]
Let $T$ be the set of rational primes $p$ inert in $L$ and 
unramified for $\rho_2$.  The
characteristic polynomials of $\rho_1(\sigma_p)$ and
$\rho_2(\sigma_p)$ for $p\in T$, are respectively of the form $(X-p)^2$ and
$X^2+p^2$. Hence for $i=1, 2$, $\rho_i(\sigma_p)^4=p^4I$, where $I$ is
the identity matrix. Thus these representations are potentially
equivalent at $T$. 

The algebraic monodromy group $G_1$ of $\rho_1$ is isomorphic to
$\G_m$. For any natural number $m$, the characters $\psi^m$ and
$(\psi^{\tau})^m$ are not equal up to multiplication by a root of
unity, by considering the values of the character at places  of degree
one over $\Q$. It follows that the algebraic monodromy group  $G_2$ is
isomorphic to the normalizer of the Weil restriction of scalars
$R_{L/\Q}(\G_m)\subset GL_2$. Thus $G_1$ is connected, the
representations $\rho_1$ and $\rho_2$ are potentially equivalent at a
set of places of density $\frac{1}{2}$, but are not potentially
equivalent yielding a contradiction to Theorem 2.1 proved  in
\cite{PR}. 
\end{example}

\begin{example} \label{example:general}
Example \ref{example-serre} can be generalized. Let $Z\simeq \G_m^k$ be a split torus of dimension $k$ over a field $E$. The automorphism group of $\G_m^k$ can be identified with $GL(k,\Z)$. Let $\theta\in \mbox{Aut}(Z)$ be an automorphism of finite order, such that the $\theta$-invariants of $Z$ is a finite group of order $m$. Form the semi-direct product $Z<J>=Z\rtimes \Z/n\Z$, where the  
generator $1\in \Z/n\Z$ is denoted by $J$. The multiplication map is defined as, 
$(xJ^k)(yJ^l)=xJ^kyJ^{-k}J^{k+l}=x\theta^k(y)J^{k+l}$, where $x, y\in Z$. In particular, $JxJ^{-1}=\theta(x)$. 

Let $\rho$ be an absolutely irreducible representation of $ Z<J>$ of dimension greater than $1$. The image of $Z$ cannot be trivial, for otherwise, $\rho$ will factor via the group $\Z/n\Z$. Thus $\rho$ cannot be potentially trivial. On the other hand, it can be seen (see  \ref{subsection:examplegeneral}), that the elements of the coset $ZJ$ are of order at most $mn$. 

\end{example}

\begin{example}
Once there is a pair $(\rho_1, \rho_2)$ of representations of $G_K$
which fails to be potentially equivalent, then any twist
$(\rho_1\otimes \eta, \rho_2\otimes \eta)$, where $\eta$ is a finite
dimensional linear
representation of $G_K$, satisfies the hypothesis of being
locally potentially equivalent at $T$. Since the character of $\eta$ is 
non-vanishing in some neighbourhood of identity, the twists are not
potentially equivalent either. 

Further, if $K\supset K_0$, and the representations $\rho_1\otimes \eta$ and
$\rho_2\otimes \eta$ are respectively restrictions to $G_K$ 
of representations $\eta_1$ and
$\eta_2$ of $G_{K_0}$, then the pair of representations $(\eta_1,
\eta_2)$ will also provide a counterexample. These representations
will be locally
potentially equivalent at the unramified set of places $T_0$ of $K_0$
that lie below $T$, but will not be potentially equivalent.

\end{example}

\subsection{Theorems}
Our main theorem says that  the failure of $\rho_1$ and $\rho_2$
being potentially isomorphic stems from 
the presence of a summand generalizing  the
above examples, where upto a twist the first representation is trivial and the
second representation  is abelian. 

Given a group $\Gamma$ and a representation $\rho: \Gamma\to GL(V)$, there is a 
finite extension $E$ of $F$, over which there is an isotypical decomposition over $E$, 
\begin{equation}\label{eqn:isotyp1}
(\rho, V) \simeq \oplus_{i=1}^t (\rho_{i}, V_i)
\end{equation}
where $\rho_i$ are representations of $\Gamma$ to $GL(V_i)$ of the form, 
\begin{equation}\label{eqn:isotyp2}
\rho_{i}\simeq r_i\otimes \rho_{i}'.  
\end{equation} 
Here $r_i$ is an absolutely irreducible representation of $\Gamma$, and $\rho_i'$ is a trivial representation of $\Gamma$ of dimension $n_i/\mbox{dim}(r_i)$, where $n_i=\mbox{dim}(V_i)$. 

\begin{theorem}\label{theorem:pe}
Let $K$ be a global field and $S$ be a set of places of $K$ 
containing the archimedean places and of density zero.
Let $F$ be a non-archimedean local field of
characteristic zero, and $\rho_1, \rho_2: G_K\to GL_n(F)$ be
semisimple continuous representations unramified outside $S$,
satisfying  the following hypothesis:
\begin{description}
\item[H1] The algebraic monodromy group $G_1$ of $\rho_1$ is
  connected.

\item[H2]  There exists a set of places $T$ of $K$ disjoint from $S$
  having positive upper density such that the representations $\rho_1$
  and $\rho_2$ are locally potentially equivalent at $T$, i.e., for
  each $v\in T$, there exists a natural number $m_v\geq 1$ such that the
  $\rho_1(\sigma_v)^{m_v}$ and
  $\rho_2(\sigma_v)^{m_v}$ are conjugate in $GL_n(F)$, where
  $\rho_1(\sigma_v)$ and $\rho_2(\sigma_v)$ are respectively the Frobenius
  conjugacy classes at $v$ of $\rho_1$ and $\rho_2$. 
\end{description}
Let $E$ be a finite extension of $F$ such that $\rho_1$ has an isotypical decomposition as given by Equations (\ref{eqn:isotyp1}) and (\ref{eqn:isotyp2}):
\[ \rho_1\simeq \oplus_{i=1}^t \rho_{1,i}\simeq \oplus_{i=1}^t r_i\otimes \rho_{1,i}'.\]
Then there exists a finite extension $L$ of $K$,  a set $T_L$ of places  of $L$ having positive density lying above
the set of places $T$ of $K$, 
and decomposition
\[ \rho_2\mid_{G_L}\simeq \oplus_{i=1}^t \rho_{2,i}, \]
where $\rho_{2,i}$ are semisimple representations of $G_L$ to
$GL(n_i, E)$ satisfying the following:

\begin{enumerate}
\item For each $i=1, \cdots, t$, there exists a 
 representation $\rho_{2,i}'$ of $G_L$ into
$GL_{n_i/\mbox{dim}(r_i)}(E)$  such that 
\[\rho_{2,i}\simeq r_i\mid_{G_L}\otimes \rho_{2,i}'.\]
\item The representation $\rho_{2,i}'$ factors as a representation  
$G_L \to Z<J>(E)\to GL_{n_i/\mbox{dim}(r_i)}(E)$, for some group of the form $Z<J>$ as in Example \ref{example:general}. 
\item The  trivial representation $\rho_{1,i}'$ and $\rho_{2,i}'$ of $G_L$ are
  locally potentially equivalent at $T_L$. 

\item If either $r_i$ or $\rho_{2,i}'$ is non-trivial, then the representations  $r_i$ and  $\rho_{2,i}'$ are non-isomorphic. 

\end{enumerate}

\end{theorem}

\begin{remark} In the conclusion of this theorem, a restriction to a
  finite extension $L$ of $K$ is part of the conclusion. However, 
it is not possible to remove
the phrase `potential', without sacrificing the positive density of
the set of places $T_L$, as can be seen from Example
\ref{Serre-arithexample}. 
\end{remark}

\begin{remark} The decomposition $\rho_{k,i}\simeq r_i \otimes \rho'_{k,i}$ for $k=1,2$ is not necessarily a motivic decomposition even if $\rho_i$ are motivic. Here we consider  a representation to be motivic if it satisfies some integrality and purity conditions for the eigenvalues of the Frobenius classes at unramified places.
\end{remark}

We now present a different version of the foregoing theorem in terms
of character values evaluated at a fixed  power of Frobenius classes: 

\begin{theorem}\label{theorem:main}
  With notation as in Theorem \ref{theorem:pe}, assume that the
following hypothesis ${\bf H2'}$ is satisfied instead of {\bf H2}:

\begin{description}

\item[${\bf H2'}$] There exists a set of places $T$ of $K$ disjoint from $S$
  having positive upper density, and a natural number $m$ such that 
for $v\in T$, 
\begin{equation}\label{h2}
\mbox{Tr}(\rho_1(\sigma_v)^m)=\mbox{Tr}(\rho_2(\sigma_v)^m).
\end{equation}

\end{description} 
Then the conclusion of Theorem \ref{theorem:main} remain valid, except
that Condition (3) of the conclusion should be replaced by the
condition, 
\[\mbox{Tr}(\rho_{1,i}'(\sigma_w)^m)=\mbox{Tr}(\rho_{2,i}'(\sigma_w)^m),
  \quad \forall w \in T_L.\]
 
\end{theorem}
\begin{remark} Theorem \ref{theorem:main} differs from Theorem
  \ref{theorem:pe}, in that apart from bounding $m_v$ uniformly independent of
  $v\in T$, we are able to work with character values of the Frobenius
  classes, rather than the
  $m$-th powers of the Frobenius classes being conjugate. We don't
  know whether we can replace ${\bf H2'}$ with the following
  condition: 
\[\mbox{Tr}(\rho_1(\sigma_v)^{m_v})=\mbox{Tr}(\rho_2(\sigma_v)^{m_v}),\]
where $m_v$ are natural numbers depending on $v$. 
\end{remark}

\subsection{Potential equivalence}
The examples given above suggest the possibility of 
proving potential equivalence
of $\rho_1$ and $\rho_2$ (assuming $G_1$ is connected and upper
density of $T$ is positive), 
under some additional natural hypothesis on the nature of the
representations $\rho_1$ and $\rho_2$. 

Given a representation $\rho: G_K\to GL_n(F)$ as above, consider the commutant
algebra for any extension field $E$ of $F$, 
\[ C_K(\rho, E):=\{X\in M_n(E)\mid X\rho(g)=\rho(g)X \quad \forall g\in G_K\}.\]
When the representation $\rho$ arises is the Galois representation $\rho_A$ attached to the Tate module of an abelian variety, Faltings theorem proving 
Tate's conjecture asserts
\[ C_K(\rho, F)= \mbox{End}_K(A)\otimes F.\]
The algebra $C_K(\rho, F)$ is also the commutant of the group $G_{\rho}(F)$ in $M_n(F)$. Since the characteristic of $F$ is zero, $C_K(\rho, E)=C_K(\rho, F)\otimes_F E$. For any finite extension $L$ of $K$, $C_K(\rho, F)\subset C_L(\rho, F)$. The commutant algebra stabilizes for $L$ sufficiently large: 
by Zariski density, whenever the image $\rho(G_L)$ is contained inside $G_{\rho}^0(F)$, where $G_{\rho}^0$ is the connected component of 
$G_{\rho}$. We call this the stable commutant algebra and denote it by $C_{\bar{K}}(\rho, F)$. 

As corollaries of the theorems
stated above, we obtain the
following theorem providing instances when (global) potential equivalence can be deduced: 
\begin{theorem}\label{theorem:appln}
With hypothesis as in Theorem \ref{theorem:main} (or Theorem
\ref{theorem:pe}),  assume further that
either of the following conditions hold:
\begin{enumerate}

\item $\rho_1$ is absolutely irreducible. 
\item The algebraic monodromy group $G_2$ of $\rho_2$ is connected. 
\item The algebraic ranks of $G_1$ and  $G_2^0$ are equal,
  where $G_2^0$ is  the connected component
  of identity of $G_2$. 
\item For any finite extension $L$ of $K$, 
 \[ C_K(\rho, F)=C_L(\rho, F)=C_{\bar{K}}(\rho, F).\]

\end{enumerate}
Then $\rho_1$ and $\rho_2$ are potentially equivalent. Further, when
$\rho_1$ is absolutely irreducible, there exists a character
$\chi:G_K\to \bar{F}^*$, such that $\rho_2\simeq \rho_1\otimes \chi$.
\end{theorem}

The last conclusion is deduced by an argument using Schur's lemma
(\cite{Ra}). One cannot expect that the character $\chi$ has values
in $F^*$ in general, as a pair of representations $\rho$ and
$\rho\otimes \chi$, with $\rho$ absolutely irreducible and $\chi$ an
arbitrary character with values in $\bar{F}^*$ satisfy the hypothesis
of the corollary. 

\begin{remark}
In general, one can possibly expect the following principle to be
valid: suppose  a property $P$ of the
conjugacy classes of $GL_n\times GL_n$ is given, 
such that a pair of $\ell$-adic representations
of $G_K$ become potentially equivalent whenever the 
Frobenius classes satisfy property $P$ at a set of places of density
one. Then the expectation, barring few exceptions, is that potential
equivalence will hold under the weaker assumption that the  
Frobenius classes satisfy property $P$ at a set of places of positive
density, provided one of the algebraic monodromy groups is connected. 
It is to be hoped that the algebraic  method outlined out here 
can be applied to answer such  questions.  
\end{remark}

\subsection{Applications} 
The above corollary can be applied to the Galois representations
attached to modular forms, and we deduce, 
\begin{corollary} Let $f$ and $g$ be two newforms of level
  $N_1$ and \( N_2 \), and respectively of weights $k_1$ and $k_2$. 
Suppose that at a set $T$ of primes $p$ coprime to $N_1N_2$ of positive
upper density, and natural numbers $n_p, ~p\in T$,  
\[ \mbox{Tr}(T_p^{n_p}(f))= \mbox{Tr}(T_p^{n_p}(g)),\]
where $T_p$ is the Hecke operator at $p$.  Assume
  further that one of the forms is not CM.  Then the weights of the
  two forms are equal and there exists a Dirichlet character
$\chi$ such that for $p\not|N_1N_2$, 
\[ a_p(f)=\chi(p)a_p(g),\]
where for a newform $h$, $a_p(h):= \mbox{Tr}(T_p(h))$ is the $p$-th
Hecke eigenvalue of $h$. 
\end{corollary}
Indeed the corollary follows from either Part (i), (ii) or (ii) of Theorem \ref{theorem:appln}, from well known properties of the Galois representation attached to a non-CM newforms. 

One of the motivating problems for the questions considered here, is
the following application to abelian varieties, which can be deduced
from Part (iv) Theorem \ref{theorem:appln},  and the theorem of
Tate, Zarhin and Faltings proving Tate conjecture on isogenies of abelian
varieties: 
\begin{corollary}\label{cor:kumar-vijay}
Let $A, ~B$ be Abelian varieties of dimension $g$
defined over a number field
$K$ without complex multiplication (over $\bar{K}$).  
Suppose that $T$ is a set of finite places of positive upper density of $K$ 
consisting of places
of good reduction for $A$ and $B$,  such that for every
$v\in T$, the reduction \( A_v, ~ B_v \) modulo $v$  of the Abelian
varieties $A, ~B$ at $v$, are isogenous over a finite extension of
the residue field $k_v$ of \( K \) at \( v \). 

Assume further that the 
algebraic monodromy group attached to the Galois representation on the 
$\ell$-adic Tate module of $A$ is connected, and $\mbox{End}_K(B)=\mbox{End}_{\bar{K}}(B)$, i.e., all the endomorphisms of $B$ are defined over $K$. 

Then $A$ and $B$ are isogenous over a finite
extension $L$ of $K$. 
\end{corollary}
\begin{remark} It would be interesting to know whether the statement will hold for abelian varieties in general. 
\end{remark}

\subsection{Outline of the proof}
Without any hypothesis on the algebraic monodromy
  groups, potential equivalence can be deduced by assuming that the
  density of  $T$ is one (\cite{PR}). 
The main thrust of this paper, is to deduce
 a similar conclusion, assuming only that density of $T$ is positive. 
The proof of the Theorems \ref{theorem:pe} and \ref{theorem:main}
follow along the lines of
(\cite{Se}) and (\cite{Ra}), by considering the algebraic monodromy
groups of the representations involved and converting the problem to 
one on algebraic groups by means of an algebraic Chebotarev density
theorem (see Theorem \ref{Algebraic-Chebotarev}). 
The algebraic Chebotarev density theorem assures the existence of a
connected component of the monodromy group of $\rho_1\times \rho_2$ 
where the representations are potentially equivalent.

If we assume
further that one of the monodromy groups is connected,
the algebraic formulation allows us to base change to complex
numbers and to employ a version of the unitary trick similar 
to the one used in \cite{Ra}. 
The crucial and new observation out here is to interpret the
consequence of the unitary trick and convert
the problem to one involving the first representation and twist of the
second representation by an automorphism of finite order (see Equations \ref{eqn:main}, \ref{eqn:theta}, \ref{eqn:thetatwist}). 
An appeal to a classical theorem on fixed points of finite order automorphisms of semisimple groups (Theorem \ref{theorem:fixedpoints})  allows us to get at
the algebraic structure of the representations satisfying the hypothesis
of Theorem \ref{theorem:main}.   

The method of proof allows us to conclude (see Theorem \ref{theorem:imagepower})
that the images of connected
components of semisimple algebraic groups with respect to the power
maps $P_m$ will not collapse. This is unlike the situation given by Examples \ref{example-serre} and \ref{example:general}, where collapsing happens when the connected component is a torus.

\section{Algebraic formulation of Theorem \ref{theorem:main}}
\subsection{A uniform bound for the exponents in 
 Theorem \ref{theorem:pe}}
We first show that the exponents $m_v, ~v\in T$ appearing in Theorem
\ref{theorem:pe} can be uniformly bounded,  using the fact that there are only
finitely many roots of unity in any non-archimedean local field: 
\begin{lemma}\label{lemma-finiterootsof1} Let \( \sigma_1 \) and \( \sigma_2 \)
be two semisimple elements in \( GL_n ( F ) \) where \( F \) is local
field (finite extension of \( \Q_\ell \)). Suppose there exists a
non-zero integer \( k \) such that \( \sigma_1^k \) and \( \sigma_1^k
\) are conjugate in \( GL_n (F) \). Then, there exists a positive
integer \( m \) depending only on \( n \) and \( F \) such that
$\sigma_1^m$ and $\sigma_2^m$ are conjugate in \( GL_n (F) \).
\end{lemma} 
\begin{remark} Since we are working in \( GL_n \), two elements are
conjugate in \( GL_n ( F) \) if and only if they are  conjugate in \(
GL_n ( \bar{F} ) \). 
\end{remark}

\begin{proof} Choose an algebraic closure \( \bar{F} \) of \( F
\). Let \( F' \) be the extension of \( F \) in \( \bar{F} \)
generated by the eigenvalues of \( \sigma_1 \) and \( \sigma_2 \) in
\( \bar{F} \). It is easy to see that 
\[ [ F' : F ] \leq ( n! )^2 .
\]  The number of roots of unity contained in such a field \( F' \) is
bounded above by some positive integer \( m_0 \) depending only on  \(
[ F' : \Q_\ell ] \), thus depending only on $n$ and \( [ F : \Q_\ell ]
\). 

Let \( \{ \alpha_1, \cdots, \alpha_n \} \)  (respectively \( \{
\beta_1 , \cdots, \beta_n \} \)) be the eigenvalues of \( \sigma_1 \)
(respectively \( \sigma_2 \)). Since by our hypothesis \( \sigma_1^ k
\) is conjugate of \( \sigma_2^k \) we have up to a permutation, 
\[ \alpha_i^k  = \beta_i^k , ~ \qquad \forall ~ 1 \leq i \leq n .
\] Hence \( \alpha_i \) and \( \beta_i \) differ by a root of unity,
which lies in \( F' \).  Thus from the above comment, for \( m = m_0!
\) we have: 
\[ \alpha_i^m  = \beta_i^m , ~ \qquad  \forall ~ 1 \leq i \leq n .
\] But since both \( \sigma_1 \) and \( \sigma_2 \) are semisimple
elements in \( GL_n ( \bar{F} ) \), \( \sigma_1^m \) and \( \sigma_2^m
\) are conjugate in \( GL_n (F) \). 
\end{proof}

It follows from this lemma, upon  assuming the hypothesis of
Theorem \ref{theorem:pe}, there exists a positive integer \( m \)
independent of \( v \in T \), and such that for all  \( v \in T \),
\( \rho_1 ( {\sigma_v})^m \) and \( \rho_2 ( {\sigma_v})^m \) are
conjugate in \( GL_n ( F ) \). 

\subsection{An application of an algebraic Chebotarev density theorem}
We recall Theorem 3
of \cite{Ra}, an algebraic interpretation of results proved in
Section 6 (especially Proposition 15) of \cite{Se}, giving an
algebraic  formulation of the Chebotarev density theorem 
for the density of places satisfying an algebraic conjugacy condition:

\begin{theorem}\label{Algebraic-Chebotarev}\cite[Theorem 3]{Ra}~~
Let \( M \) be an algebraic group defined over a $l$-adic local field
\( F \) of characteristic zero.  Suppose 
\[
\rho : G_K \rightarrow M(F) 
\]
is a continuous representation unramified
 outside a finite set of places of $K$. 

Suppose  \( X \) is a closed subscheme of \( M\) 
defined over \( F \) and stable under the adjoint action of \( M \) 
on itself. Let 
\[
C := X(F) \cap \rho ( G_K ) .
\]
Let \( \Sigma_u \) denote the set of finite places of \( K \)
 at which \( \rho \) is unramified. 
Then the set 
\[
S : =  \{ v \in \Sigma_u  ~| ~ \rho ( \sigma_v ) \subset C \}.
\]
has a density given by 
\[
d(S) = \frac{ | \Psi | }{ | \Phi | }, 
\]
where $\Phi$ is the set of connected components of $G$, and 
$\Psi $ is the set of those $\phi \in \Phi$ such that the
corresponding connected component $G^\phi$ of $G$ is contained in $X $.

\end{theorem}

\begin{corollary}\label{semisimple} Let $\rho$ be a semisimple
continuous $\ell$-adic representation of $G_K$ to $GL_n(F)$ unramified
outside a finite set of places of $K$.  Then there is a density one
set of places of $K$ at which $\rho$ is unramified and the
corresponding Frobenius conjugacy class is semisimple. 
\end{corollary}
\begin{proof} Since the  representations are assumed to be
  semi-simple,  the
algebraic monodromy groups are reductive algebraic group defined over
\( F \).  The corollary follows from the fact that 
the semisimple elements in a reductive group $G$ 
contain a Zariski dense open subset of $G$. 
\end{proof}

Let $m$ be a natural number. 
Consider the following Zariski closed, invariant subsets of $GL_n \times
GL_n$:
\[\begin{split}
X_m &:= \{ ( g_1 , g_2 ) \in GL_n \times GL_n ~| ~ {\rm{Trace}} (g_1^{m}
) = {\rm{Trace}} (g_2^{m} )\}\\
Y_m &:= \{ ( g_1 , g_2 ) \in GL_n \times GL_n ~| ~ {\rm{Trace}} (\Lambda^j(g_1^{m})
) = {\rm{Trace}} (\Lambda^j(g_2^{m} )), \quad j=1,\cdots, n\},
\end{split}
\]
where $\Lambda^j$ denotes the $j$-th exterior power representation of
$GL_n$. For semisimple elements, the condition that $(g_1,g_2)\in Y_m$,
is equivalent to saying that $g_1^m$ and $g_2^m$ are conjugate. 

Given the hypothesis of Theorem \ref{theorem:main} 
(resp. Theorem \ref{theorem:pe}), we have 
\[ \rho ( \sigma_v ) \in X_m (F) \quad  (\mbox{resp.} ~~\rho ( \sigma_v ) \in Y_m (F))
\quad \quad ~ \forall v \in T' , 
\] where $\rho=\rho_1\oplus \rho_2: G_K\to GL_n(F)\times GL_n(F)$ is
the direct sum of the representations $\rho_1$ and $\rho_2$. In what
follows we present the proof for Theorem \ref{theorem:main}, and make
remarks only as required for the proof of Theorem \ref{theorem:pe}.

Let $G$ denote the algebraic monodromy group of $\rho$. 
Since \( T \) is of positive upper density, by Theorem
\ref{Algebraic-Chebotarev} above, there exists a connected component
\( G^\phi \) of \( G \) such that \( G^\phi \) is contained in \( X_m
\).

We are led to consider the following problem in the context of
algebraic groups: let $G\subset GL_n\times GL_n$ be a reductive
algebraic group and $p_1, ~p_2$ denote the two projections. Assume
that the image $p_1(G)=G_1$ is connected. Suppose
that there is a connected component $G^{\phi}$ of $G$ contained inside
$X_m$. What can we conclude about the representations $p_1$ and $p_2$
of $G$? For the rest of this section, we will be following this
notation. 

We first observe the equivalence of the representations  
when  $G$ is connected:
\begin{proposition}\label{prop:connected}
With the above notation, suppose $G$ is connected and $G\subset
X_m$. Then the representations $p_1$ and $p_2$ are equivalent. 
\end{proposition}
\begin{proof}
Since $G$ is connected, the $m$-th power map $x\mapsto x^m$ from $G$
to $G$ is dominant. Hence $G\subset X_1$, and the representations
$p_1$ and $p_2$ are equivalent.
\end{proof}
\begin{remark}\label{remark:error}
 The surjectivity of the $m$-th power map fails when we
  consider it between connected components. In \cite{PR} 
the argument continued as follows:
  it  can be assumed that $m$ is chosen such that the $m$-th power map
  sends $G^{\phi}$ to $G^0$, the connected component of identity in
  $G$. The image will be contained inside the subvariety $X_1$. Upon
  the {\em erroneous} assumption that the $m$-th power map is surjective,
  one concludes that $G^0\subset X_1$ and this implies potential
  equivalence of $\rho_1$ and $\rho_2$. But this assumption is wrong,
  as was pointed out by J.-P. Serre. 
\end{remark}

\subsection{A unitary trick} One of the advantages with the algebraic
formulation is that it allows base change to the field of complex
numbers, which makes it amenable to transcendental methods. Choosing an
isomorphism of the algebraic closure of $F$ with $\C$, we consider the
analogous problem over complex numbers. 

Let $U$ be a maximal compact subgroup of $G(\C)$ which we will assume is
contained inside $U(n)\times U(n)$, where $U(n)\subset GL_n(\C)$ is
the (standard) group of unitary $n\times n$ matrices. The group $U$ is
Zariski dense in $G$. Hence the intersection $U^{\phi}:=U\cap
G^{\phi}$ is Zariski dense in $G^{\phi}$. In particular, it is
non-empty.

The image $p_1(U)$ of the projection of $U$ to the first factor is a
maximal compact subgroup of $G_1(\C)$. Since by hypothesis $G_1$ is
connected, $p_1(U)$ is connected and equal to $p_1(U^0)$, 
where $U^0:=U\cap G^0$ is the connected component of $U$. Hence for any connected component $U^{\psi}$ of $U$, $p_1(U^{\psi})= p_1(U^0)$. In particular, 
the image  $p_1(U^{\phi})$ is a subgroup of $G_1(\C)$. 

Hence there exists an element of the form
$(I_n,j)\in U^{\phi}$, where $I_n$ denotes the identity
matrix in $GL_n(\C)$.  Since $G^{\phi}\subset X_m$, 
\[ n=\mbox{Tr}(I_n^m)= \mbox{Tr}(j^m).\]
Since the only unitary matrix in with trace equal to $n$ is the identity
matrix, we conclude that $j^m=I_n$. Thus, $G^{\phi}=G^0(I_n,j)$ with
$j^m=I_n$. Let $J=(I_n,j)$. The condition $G^{\phi}\subset X_m$,
translates to the condition $(xJ)^m\in X_m$ for any $x\in G^0$. 

\begin{remark} In \cite{Ra}, we considered the case $m=1$. In this
  case, we have   that $(I_n, I_n)\in G^{\phi}(\C)$, and this implies
  that $G^{\phi}=G^0$, and the representations $p_1$ and $p_2$ are
  isomorphic restricted to $G^0$. 
 Hence we conclude that $\rho_1$ and $\rho_2$ are
  potentially isomorphic. 
\end{remark}

For $m\geq 2$, the significance of the unitary trick lies in the
following crucial observation: 
\begin{equation}\label{eqn:main}
(xJ)^m=xJ\cdots xJ=xJxJ^{-1}J^2xJ^{-2}\cdots
J^{m-1}xJ^{-(m-1)}J^{m-1}J, \quad x\in G^0.
\end{equation}
Let $\theta(x)=JxJ^{-1}$ denote the automorphism of finite order (dividing $m$) 
of $G^0$ induced by the conjugation action of $J$. Since by the
unitary trick $J^m=(I_n, I_n)$,
the above equation becomes, 
\begin{equation}\label{eqn:theta}
 (xJ)^m=x\theta(x)\theta^2(x)\cdots\theta^{m-1}(x), \quad x\in G^0.
\end{equation}
The condition  $G^{\phi}\subset X_m$ now translates to the following  the first
representation is absolutely irreducible
condition on the representations $p_1$ and $p_2$ of $G^0$: 
\begin{equation}\label{eqn:thetatwist}
\mbox{Tr}(p_1(x)^m)=\mbox{Tr}\left(p_2(x\theta(x)\theta^2(x)\cdots
\theta^{m-1}(x))\right), \quad x\in G^0.
\end{equation}

\subsection{Example \ref{example:general}} \label{subsection:examplegeneral}
The above calculation can be reversed. We put ourselves in the context of Example \ref{example:general}. For $x\in Z$, 
\[ (xJ)^n=x\theta(x)\cdots \theta^{n-1}(x),\]
is $\theta$ invariant as $Z$ is abelian.  Since, by assumption the subgroup of $\theta$-invariants of $Z$ is finite (of order $m$), the element $xJ$ is of finite order.   Hence for any representation $\rho$ of $Z<J>$, the elements of the coset $xJ$ have order at most $mn$, where $m$ is the order of the subgroup of $\theta$-invariants of $Z$. 

\begin{remark}\label{remark:noncommute}
We make a remark, which we will use in deducing Part (iv) of Theorem \ref{theorem:appln}. 
If $\rho(Z)$ and $\rho(J)$ commute, the fact that $xJ$ is of finite order implies that $\rho(x)$ is of finite order for $x\in Z$. Since $Z$ is connected, its image is trivial and this contradicts the irreducibility of $\rho$. In particular, $\rho(Z)$ and $\rho(J)$ do not commute, i.e., there exists elements of $\rho(Z)$ which are not fixed by the conjugacy action of $\rho(J)$. 
\end{remark}

\subsection{Algebraic analogue of Theorem \ref{theorem:main}}\label{sec:alganalog}
In this section, we consider the subgroup of $G\subset GL(n)\times
GL(n)$ generated by 
$G^0$ and the the connected component $G^{\phi}$, as considered above.
By an abuse of notation, we continue to
denote this subgroup by $G$. Since $J$ has finite order, 
it is semisimple with eigenvalues roots of
unity. Hence $J$ and the  automorphism $\theta$ are
 defined over a finite extension $F'$ of
$F$. The group $G$ is defined over $F$, and 
over $F'$,  is isomorphic to the group
generated by $G^0$ and the element $J=(1,j)\in G(F')$. The conjugation
action by the element $(1,j)$ induces the automorphism  $\theta$ on
$G^0$ defined over $F'$. 
We decompose  $G$ (the decompositions are valid over any finite extension of $F$ containing $F'$) 
 with respect to the action of $\theta$ as
follows: 

\begin{enumerate}
\item The derived subgroup $G'$ of $G^0$  is a
connected semisimple group defined over $F$. 

\item  The
connected component $Z$ of the center of $G^0$ is invariant under $\theta$.  

The automorphism $\theta$
leaves stable the groups $G'$ and $Z$ (considered as subgroups over
$F'$), and there is a decomposition 
\[ G^0=G'Z\quad \mbox{and}\quad G'\cap Z ~~\mbox{is finite}.\]
We further decompose $Z$ with respect to the action of
$\theta$. 

Consider the lattice $X^*(Z)$ of characters of $Z$. 
Define endomorphisms of  $X^*(Z)$, 
\[ I_{\theta}= 1-\theta\quad \mbox{ and} \quad N_{\theta}=
1+\theta+\cdots+\theta^{m-1}.\]
Let $X^{\theta}$ be the subgroup of  $X^*(Z)$ on which $\theta$ acts
trivially and $ Z^{\theta}$ be the corresponding subtorus of
$Z$ defined over $F'$.
This is the maximial subtorus of $Z$ on which $\theta$ acts
trivially. 

\item Decompose the  space  $X^*(Z)\otimes \Q= X^{\theta}\otimes \Q\oplus
Y$, where $Y$ is the kernel of  $N_{\theta}$. Choose a lattice $L_0\subset Y$
and let $L_{\theta}=\sum_{i=0}^{m-1}\theta^i(L_0)$. The lattice
$L_{\theta}$ is $\theta$-invariant and is the character group of a
$\theta$-stable subtorus $Z_{\theta}$ of $Z$. 
The invariants of $\theta$ acting on $Z_{\theta}$ and
$G'Z^{\theta}\cap Z_{\theta}$ are  finite groups. 

\item  Let $Z_{\theta}<\!\!J\!\!>$ 
be the subgroup of $G$ generated by 
$Z_{\theta}$ and the element $J$.  
The groups $G'Z^{\theta}$ and $Z_{\theta}<\!\!J\!\!>$ are normal in
  $G$,  and there is a decomposition 
\[ G=G'Z^{\theta} (Z_{\theta}<\!\!J\!\!>)\quad \mbox{and}\quad 
G'Z^{\theta}\cap Z_{\theta}<\!\!J\!\!> ~~\mbox{is finite}.\]

\end{enumerate}

We have the following algebraic analogue of Theorem \ref{theorem:main}:
\begin{theorem} \label{theorem:main-alg}
Let $G$ be as above.   Suppose that the following condition is satisfied for
$x\in G^0$:
\begin{equation}\label{eqn:tracem}
 \mbox{Tr}(p_1(x^m))=\mbox{Tr}(p_2(x\theta(x)\cdots
  \theta^{m-1}(x))).
\end{equation}
With respect to the decomposition $G=G'Z^{\theta}  Z_{\theta}<\!\!J\!\!>$ given
above, the following hold: 
\begin{enumerate}
\item $\theta$ acts trivially on $G'Z^{\theta}$, and 
\[ p_1\!\mid_{G'Z^{\theta}}\simeq  p_2\!\mid_{G'Z^{\theta}}\]

\item $ p_1$ restricted to $Z_{\theta}<\!\!J\!\!>$ is trivial. 

\item Let $E$ be a finite extension of $F'$, over which the representation
$p_1$ of $G$ has an isotypical decomposition as given by Equations (\ref{eqn:isotyp1}) and (\ref{eqn:isotyp2}): 
\[ p_1\simeq \oplus_{i=1}^t p_{1,i} \simeq \oplus_{i=1}^t R_i\otimes p_{1,i}', \]
where the representations $p_{1,i}'$ are trivial of appropriate degree. The representation $p_2$ decomposes over $E$ as 
\[p_2\simeq \oplus_{i=1}^t p_{2,i}, \]
where $p_{2,i}$ are  representations of $G$ to
$GL_{V_i}$. 
For $1\leq i\leq t$, there exists 
representations $p_{2,i}'$ of $Z_{\theta}<\!\!J\!\!>$ into
$GL_{n_i/\mbox{dim}(R_i)}(E)$,  satisfying: 
\begin{enumerate}
\item As representations of the group $G'Z^{\theta}\times
  Z_{\theta}<\!\!J\!\!>$, 
\[p_{2,i}=R_i\otimes p_{2,i}'.\]
where  $R_i$ is considered as a representation of $G'Z^{\theta}$. 
\item The representations  $p_{1,i}' , ~p_{2,i}'$ satisfy, 
\begin{equation}\label{eqn:tracem-icomp}
n_i= \mbox{Tr}(p_{1,i}'(x^m))=\mbox{Tr}(p_{2,i}'(x\theta(x)\cdots
  \theta^{m-1}(x))), \quad x\in Z_{\theta}.
\end{equation}
Equivalently, for $y=xJ$ belonging to the coset $Z_{\theta}J$, 
\[n_i= \mbox{Tr}(p_{1,i}'(y^m))=\mbox{Tr}(p_{2,i}'(y^m)).\]
\end{enumerate}
\end{enumerate}
\end{theorem}

\begin{remark} A similar analogous statement can be made for Theorem
  \ref{theorem:pe}, where the hypothesis is modified by considering
  the analogue of Equation (\ref{eqn:tracem}) for all exterior
  powers. Similarly the conclusion can be strengthened to say that
  the equality of Equation (\ref{eqn:tracem-icomp}) holds for all
  exterior powers. Since $Z_{\theta}$ is a torus, the elements in the
  image of  $Z_{\theta}<\!\!J\!\!>$ will be semisimple for any linear
  representation of  $Z_{\theta}<\!\!J\!\!>$. From this and the equality of
  traces for all exterior powers, one concludes 
  that the elements $p_{1,i}'(y^m)$ and
  $p_{2,i}'(y^m)$    are conjugate for $y=xJ\in Z_{\theta}J,
  ~x\in Z_{\theta}$.  
\end{remark}

\subsection{Proof of Parts (1) and (2) of Theorem 
\ref{theorem:main-alg}} 
Let $H$ denote the connected component of identity of 
$\mbox{Ker}(p_1\!\mid_{G'})$. 
The group $H$ is semisimple, and since $p_1=p_1\circ \theta$,  $H$ is
$\theta$-stable. Let  $H^{\theta}$ denote the identity connected component 
of the fixed points of $\theta$ acting on $H$. 

Restricted to
$H^{\theta}$, Equation (\ref{eqn:tracem}) gives the identity: 
\[ n=\mbox{Tr}(p_1(x^m))=\mbox{Tr}(p_2(x^m)),\quad x\in H^{\theta}, \]
where we have used the fact that 
$p_1$ restricted to $H$ is trivial. By Proposition
\ref{prop:connected}, the representation $p_2$ restricted to
$H^{\theta}$ is trivial. 
Since $G'\subset GL_n\times
GL_n$, the representation $p_2$ restricted to $H$, and in particular
to $H^{\theta}$ is injective. Hence $H^{\theta}$ is trivial. 

The following theorem is classical (see for instance
\cite[Chapter 8]{K}) or \cite{BM}), and we give a proof in our context
 in the following section):
\begin{theorem}\label{theorem:fixedpoints}
Let $H$ be a (non-trivial) connected semisimple group over $F$, 
and let $\theta$ be a finite order automorphism of $H$. Then
the connected component $H^{\theta}$ of the fixed points 
of $\theta$ is a non-trivial
reductive group. 
\end{theorem}
From this theorem and the triviality of $H^{\theta}$, it follows that 
$H$ is trivial, i.e., the kernel of $p_1\!\mid_{G'}$ is
finite. From the equality $p_1\circ \theta=p_1$, we get for $g\in G'$
that $\theta(g)=zg$ for some $z$ in kernel of $p_1$, i.e.,
$\theta(g)g^{-1}=z$.  The finiteness of
the kernel and the connectedness of $G'$ implies that $\theta$ acts
trivially on $G'$. 

Thus $\theta$ acts trivially on $G'Z^{\theta}$. Proposition
\ref{prop:connected} applied to Equation (\ref{eqn:tracem}) restricted
to $G'Z^{\theta}$ yields Part (1) of Theorem \ref{theorem:main-alg}. 

The proof of Part (2) of Theorem \ref{theorem:main-alg}, follows from
the fact that $\theta$ acts trivially on the first co-ordinate. Since
$Z_{\theta}$ has finite  $\theta$-invariants, being connected, it cannot have a
non-trivial quotient with trivial action of $\theta$. The element $J$
projects trivially to the first component by our choice of $J$. 

\subsection{Part (3) of Theorem \ref{theorem:main-alg}}
 By Part (1), after a change of
basis of one of the representations we can assume that 
\[ p_1\!\mid_{G'Z^{\theta}}= p_2\!\mid_{G'Z^{\theta}}\]
Let $E^n=\oplus_{i=1}^t V_i$ be the decomposition into isotypical
components of absolutely irreducible representations of 
the representation $ p_1\!\mid_{G'Z^{\theta}}$. These
components  are stabilized  by $G$ as
$G'Z^{\theta}$ is normal in $G$. 
Hence, considered as representations of the group $G'Z^{\theta}\times
Z_{\theta}<\!\!J\!\!>$, the representation $p_{2,i}$ can be
written as, 
\[ p_{2,i}\simeq \oplus_{j=1}^{n_j} R_i\otimes \eta_{ij},\]
for some collection of absolutely irreducible representations 
$\eta_{ij}$ of $Z_{\theta}<\!\!J\!\!>$. Define
\[ p_{2,i}'=\oplus_{j=1}^{n_j}\eta_{ij}.\]
This gives the decomposition in Part (3.a).

To prove Part (3.b), in terms of the decompositions, 
Equation (\ref{eqn:tracem}) can be written as, 
\[ \sum_{i=1}^t\mbox{Tr}(R_i(g^m))\mbox{Tr}(p_{1,i}'(x^m))=
\sum_{i=1}^t\mbox{Tr}(R_i(g^m))\mbox{Tr}(p_{2,i}'(x\theta(x)\cdots
  \theta^{m-1}(x))),\]
for $g\in G'Z^{\theta}$ and $x\in Z_{\theta}$. Since $G'Z^{\theta}$ is
connected, the map $g\mapsto g^m$ is surjective. The representations $R_i$ are
distinct irreducible representations of $G'Z^{\theta}$. Hence by the linear
independence of characters, we obtain Part (3.c) of Theorem
\ref{theorem:main-alg}. 

\subsection{Proof of Theorem \ref{theorem:fixedpoints}}
For the sake of completeness of exposition, we give a proof of Theorem
\ref{theorem:fixedpoints} in the following context which is sufficient
for our purpose (see \cite[Chapter 8]{K} for a more detailed and
complete exposition): we take $F=\C$, $G\subset GL(n,\C)$ 
and the automorphism $\theta$ is
induced by conjugation by an unitary matrix. In this case, $\theta$
commutes with the Cartan involution $\theta_c: A\mapsto
^t\!\bar{A}^{-1}$. Hence the Cartain involution fixes  $H^{\theta}$,
and it follows that $H^{\theta}$  is a
reductive group. 

To see that  the connected component $H^{\theta}$ of the fixed points 
of $\theta$ is  non-trivial, it suffices to work 
with a compact form of $H$. Let  ${\frak H}$ denote the
complexification of the Lie algebra of the compact form. We continue
to denote by $\theta$ the inducted action on the Lie
algebra. Decompose ${\frak H}=\oplus_{i=0}^{m-1}{\frak H}_i$ as
eigenspaces for the action of $\theta$ on ${\frak H}$, where 
${\frak H}_j=\{X\in {\frak H}\mid \theta(X)=\zeta_m^jX\}$. Here
$\zeta_m$ is a fixed choice of a primitive $m$-th root of unity. 

Let $I$ be the set of indices for which ${\frak H}_i$ is non-zero, and
let $k\in I$ be such that $\zeta_m^k$ is a generator for the subgroup of
the roots of unity
generated by $\zeta_m^j, ~j\in I$. For $X\in {\frak H}_k$ and $Y\in
{\frak H}_l$, 
\[
  \theta(ad(X)(Y))=\theta([X,Y])=[\theta(X),\theta(Y)]=\zeta_m^{k+l}[X,Y].\]
Suppose $0$ is not in $I$. For any $n$, $ad(X)^n(Y)$ belongs to
$\zeta_m^{nk+l}$-eigenspace of $\theta$. The assumption on $k$ implies
that for some $n$, $ad(X)^n$ will annihilate the space ${\frak
  H}_l$. Choosing $n$ appropriately, this implies that $ad(X)$ is
nilpotent on ${\frak H}$. 
 But then there are no such elements in ${\frak H}$. Hence
$0\in I$ and this implies that $H^{\theta}$ is non-trivial. 

\subsection{Images of connected components with respect to power maps}
Examples  \ref{example-serre} and \ref{example:general}  gives instances when the power map
$P_m: x\mapsto x^m$ from a connected component $G^{\phi}$ of $G$ to
the connected component of identity has image a singleton set. Such
instances occur when the connected component is a tori. 
 The observation given by Equation
(\ref{eqn:main}) allows us to conclude that the images of connected
components of semisimple algebraic groups with respect to the power
maps $P_m$ will not collapse. 
\begin{theorem}\label{theorem:imagepower}
Let $G$ be a semisimple algebraic group over 
an algebraically closed field $F$ of characteristic zero. 
 Let $G^{\phi}$ denote a connected
component of $G$. Suppose that the order of $G^{\phi}$ is $n$. For a
sufficiently large multiple $m$ of $n$, the image of  $G^{\phi}$ by
the power map $P_m: x\mapsto x^m$ contains a (Zariski open) neighbourhood of identity of a 
non-trivial connected reductive group. 
\end{theorem}
\begin{proof} 
We show that there exists a torsion element $j$ in $G^{\phi}$ of order
$m$. Elements $x$ of $G^{\phi}$ induce automorphisms $\theta_x$  of
$G^0$ by conjugation, and two such automorphisms are related by an
inner automorphism. By \cite[Corollary 2.14]{Sp}
the sequence 
\[ 1\to \mbox{Inn}(G^0)\to  \mbox{Aut}(G^0)\to  \mbox{Out}(G^0)\to 1,\]
splits. Since  $\mbox{Out}(G^0)$ is finite, modifying $x$ by an
element in $G^0$, we get an element $j\in G^{\phi}$, such that the
corresponding inner automorphism it defines on $G^0$ is trivial. Hence
$j$ belongs to the center of $G^0$. Since the center of $G^0$ is
finite, this yields an element $j\in G^{\phi}$ of finite order.

Let $\theta_j(x)=jxj^{-1}$ be the automorphism induced on $G^0$  
by conjugation
by $j$. From Equation (\ref{eqn:theta}), the image of $P_m$ will
contain elements of the form
$x\theta(x)\theta^2(x)\cdots\theta^{m-1}(x)$, where $x$ ranges over the
elements of $G^0$. In
particular, the image will contain elements of the form $x^m$ for
$x\in (G^0)^{\theta}$, where $(G^0)^{\theta}$ denotes the connected
component of identity of the fixed points of $\theta$ acting on $G^0$.
By Theorem
\ref{theorem:fixedpoints}, this is a non-trivial connected reductive
group. The image of the map $x\mapsto x^m$ from $(G^0)^{\theta}$ to
itself is dominant and  contains a neighbourhood of identity of
$(G^0)^{\theta}$,
 and
this proves the theorem.

\end{proof}

\section{Proof of Theorem \ref{theorem:main}}
\subsection{Algebraic envelopes}\label{rmk:algmongp}
We first make some few remarks on the relationship between representations of abstract groups and their algebraic envelopes:
\begin{enumerate}

\item Suppose $\Gamma \subset G(\bar{F})$ is a Zariski dense subgroup and $\rho$ a representation of $G$. Then the isotypical decompositions $\rho\mid_{\Gamma}$ is the restriction to $\Gamma$ of the isotypical decomposition of $\rho$. This follows from the fact that a representation $R$ of $G$ is irreducible if and only if it is irreducible restricted to $\Gamma$. 

\item Suppose $\rho: \Gamma \to GL_n(\bar{F})$ is a representation with connected algebraic monodromy group $G$. If  $r: G\to H$ is a representation of algebraic groups, then the algebraic envelope of $r\circ \rho$ is the image group $r(G)$, and hence is also a connected group. In particular if
$\rho \simeq \oplus_{i=1}^t \rho_i$, then the algebraic envelopes of $\rho_i$ are connected. 

\item Suppose $\rho_1, ~\rho_2$ are two linear representations of $\Gamma$. The tensor product representation `factors' via the direct sum: $\gamma\mapsto (\rho_1(\gamma), \rho_2(\gamma))\mapsto  \rho_1(\gamma)\otimes\rho_2(\gamma)$. 
Further $\rho_1(\gamma)\otimes\rho_2(\gamma)$ is trivial if and only if both $\rho_1(\gamma)$ and $\rho_2(\gamma)$ are trivial. Hence the algebraic monodromy group of $\rho_1\otimes\rho_2$ and $\rho_1\oplus\rho_2$ 
are isomorphic. 
\end{enumerate}

\subsection{Proof of Theorem \ref{theorem:main}}
We deduce now Theorem \ref{theorem:main} from Theorem
\ref{theorem:main-alg}. From the arguments of Sections 2.2 and 2.3,
consider the group $\tilde{G}$ generated by $G^0$ and the connected
component $G^{\phi}$ (as is done in Section \ref{sec:alganalog}). The group
$\rho^{-1}(\tilde{G}(F))$ is of finite index in $G_K$ and is of the form 
$G_{K_1}$ for some finite extension $K_1$ of $K$. We apply Theorem
\ref{theorem:main-alg} to $\tilde{G}$ and let $E$ be the field given in Part
(3) of Theorem
\ref{theorem:main-alg}. 

The natural map $\pi$ from the product of
groups 
\[\pi: \tilde{G}'Z^{\theta}(E)\times Z_{\theta}<\!\!J\!\!>(E)\to \tilde{G}(E),\]
 has finite
kernel $A$. Choose a  $\theta$-stable open subgroup $M'$
(resp. $M_{\theta}$) of  $\tilde{G}'Z^{\theta}(E)$
(resp. $Z_{\theta}<\!\!J\!\!>(E)$), such that  $M'\times M_{\theta}$
intersects $A$ only at the identity element. Let  $M_{\theta}<\!\!J\!\!>$
denote the open subgroup of  $Z_{\theta}<\!\!J\!\!>(E)$ generated by
$M_{\theta}$ and $J$. This group is Zariski dense in $Z_{\theta}<\!\!J\!\!>$. 

The group  $M'\times M_{\theta}<\!\!J\!\!>$ maps isomorphically via $\pi$ to 
its image group denoted by $M$ in
$\tilde{G}(E)$. This is an open subgroup of $\tilde{G}(E)$. The intersection
$\rho(G_{K_1})\cap M$ is of finite index in  $\rho(G_{K_1})$, and is
of the form $\rho(G_L)$ for some finite extension $L$ of $K$. Thus we
can consider the map $\rho: G_L\to \tilde{G}(E)$ to factor 
via $\tilde{G}'Z^{\theta}(E)\times Z_{\theta}<\!\!J\!\!>(E)$. Composing the maps $R_i,
p'_{1,i}$ and $p'_{2,i}$ with the map $\rho$ defines respectively 
linear representations  $r_i, \rho'_{1,i}$ and $\rho'_{2,i}$ of
$G_L$. 

By the above remark, the isotypical decomposition of the representation 
$\rho_1$ restricted to $G_L$ is the restriction of the isotypical decomposition of $\rho_1$. Since the representations $R_i$ and $p_{2,i}'$ factor via different groups, provided one of them is non-trivial, they are not isomorphic. 

We need to only check the positive density upon restriction to $L$. 
From the construction, 
the algebraic monodromy of the collection of representations $\rho'_{2,i}$ is
$Z_{\theta}<\!\!J\!\!>$. The set of places $T_L$ is defined to be those places
of $L$ such that the image of the Frobenius conjugacy class by  the
representations  $\rho'_{2,i}$ lands inside the component
$Z_{\theta}J$ of  $Z_{\theta}<\!\!J\!\!>$, and hence it is of positive density
by Chebotarev density theorem. 
Theorem \ref{theorem:main} follows now from Theorem
\ref{theorem:main-alg}. 

\section{Proof of Theorem \ref{theorem:appln}}
We now prove Theorem \ref{theorem:appln}.

(1) If $\rho_1$ is absolutely irreducible, then $t=1$ and the
dimension of the  representations $\rho_{1,1}'$ and $\rho_{2,1}'$ is
one. Since they are also potentially equivalent, the values of
$\rho_{2,1}'$ lie in roots of unity of $F$, which is a finite
set. Hence the representations are potentially equivalent.

(2) Suppose that the algebraic monodromy group of $\rho_2$ is also connected. By 
Remark \ref{rmk:algmongp}, the algebraic monodromy groups of the components $\rho_{2,i}\simeq r_i\otimes \rho_{2,i}'$ are also connected for each   $i$. Further the algebraic monodromy group of the tensor product $r_i\otimes \rho_{2,i}'$ and $r_i\oplus \rho_{2,i}'$ are isomorphic. This implies that the algebraic monodromy group of $\rho_{2,i}'$ are also connected. 

Thus we are in the situation of the hypothesis of Theorem \ref{theorem:main}, where the first representation $\rho_{1,i}'$ is trivial and the second representation $\rho_{2,i}'$ has connected algebraic monodromy group. In this case, the algebraic monodromy group of $\rho_{1,i}'\oplus \rho_{2,i}'$ is isomorphic to that of $\rho_{2,i}'$, and  hence is connected. By Proposition \ref{prop:connected}, Part (2) of Theorem \ref{theorem:appln} follows. 

(3) The projection map takes a maximal torus in the algebraic monodromy group of $\rho_2$ to a maximal torus of the algebraic monodromy group of $r_i\otimes \rho_{2,i}'$ for each $i$. The latter monodromy group is isomorphic to the monodromy of 
$r_i\oplus \rho_{2,i}'$. If the ranks of the monodromy groups of $\rho_1$ and $\rho_2$ are equal, then for
  each $i$, the representations $\rho_{2,i}'$ are finite. This means
  that the representations $\rho_{2,i}'$ are potentially trivial and hence the representations $\rho_1$ and $\rho_2$ are potentially equivalent. 
  
(4) The representations $\rho_{2,i}'$ factors via the group $Z_{\theta}<J>(E)$.  By Remark \ref{remark:noncommute}, if $\rho_{2,i}'$ has an irreducible component of dimension at least two for some $i$, then there exists an element, say $u\in Z_{\theta}(E)$, such that $\rho_{2,i}'(u)$ is not invariant by the conjugacy action induced by $J$. Since $Z$ is abelian, this means that the stable commutant algebra of $\rho_{2,i}'$ is strictly larger than the the commutant algebra of $\rho_{2,i}'$. This implies that 
$C_K(\rho_2, F)\neq C_{\bar{K}}(\rho_2, F)$. 

Hence if the commutant algebra $C_K(\rho_2, F)$ is isomorphic to the stable commutant algebra  $C_{\bar{K}}(\rho_2, F)$, the representations $\rho_{2,i}'$ breaks up as a direct sum of one dimensional representations over $\bar{E}$. In particular $p_{2,i}'(Z_{\theta})$ and $p_{2,i}'(J)$ commute,  where $p_{2,i}'$ is the representation of $_{\theta}<J>$ through which $\rho_{2,i}'$ factors. 
But as remarked in Section \ref{subsection:examplegeneral}, this implies that the representations $p_{2,i}'$ is finite, and hence $\rho_{2,i}'$ is potentially finite.

This proves Theorem \ref{theorem:appln}.

\begin{acknowledgement} Our sincere thanks to J.-P. Serre for pointing
out an error in our earlier paper (\cite{PR}), and for his continued
interest which motivated us to persist with this question. We thank
the referee for useful suggestions.

The first author thanks the School of Mathematics, Tata Institute of Fundamental
Research for excellent work environment and hospitality. The first
author also thanks Mathematical Research Impact Centric Support
(MATRICS) grant under Science and Engineering Research Board,
Department of Science and Technology, Government of India for
supporting his travels that helped with this collaboration. 
\end{acknowledgement}

\end{document}